\documentclass[12pt,draft]{amsart}

\usepackage[foot]{amsaddr}

\usepackage{mathptmx}

\usepackage{amsmath,amssymb,amsthm,cite}

\DeclareMathOperator{\val}{val}

\newtheorem{lemma}{Lemma}

\newtheorem{theorem}{Theorem}

\newtheorem{problem}{Problem}

\begin{document}

\title[Semiring identities in $B_0$]{Semiring identities in the semigroup $B_0$}

\author{Vyacheslav Yu. Shaprynski\v i}\email{vshapr@yandex.ru}

\address{Ural Federal University, Institute of Natural Sciences and Mathematics, 620000, Ekaterinburg, Russia}

\thanks{Supported by Russian Science Foundation, grant No. 23-21-00289.}



\begin{abstract}
The semigroup $B_0$ is the only, up to isomorphism, 4-ele\-ment subsemigroup of the 5-element Brandt semigroup $B_2$. Being an inverse semigroup, the semigroup $B_2$ can naturally be considered an additively idempotent semiring and $B_0$ is its subsemi\-ring. We show that the semiring $B_0$ has a finite basis of identities.
\end{abstract}

\maketitle

\section{Introduction and summary}

An \emph{additively idempotent semiring (ai-semiring)} $(S,+,\cdot)$ is an algebra with two binary operations that satisfies the following identities:
\begin{itemize}
\item[1)] $x+y\approx y+x$;
\item[2)] $(x+y)+z\approx x+(y+z)$;
\item[3)] $x+x\approx x$;
\item[4)] $(xy)z\approx x(yz)$;
\item[5)] $x(y+z)\approx xy+xz$;
\item[6)] $(x+y)z\approx xz+yz$.
\end{itemize}
In other words, $(S,+)$ is a semilattice, $(S,\cdot)$ is a semigroup, and multiplication distributes over addition.

Additively idempotent semirings have been attracting significant attention of semigroup theorists recently. In particular, a number of results was dedicated to the finite basis problem in ai-semirings. As examples of earlier results, one can mention the papers~\cite{Dolinka-07,Dolinka-09-1,Dolinka-09-2,Dolinka-09-3} on abstract semirings, \cite{Aceto-03} on tropical semirings, and \cite{Andreka-11} on semirings of binary relations. The recent article~\cite{Jackson-22} contains a number of deep results on the topic, as well as open problems. In the present paper, we consider Problem~7.7 in~\cite{Jackson-22} which deals with inverse semigroups.

Recall that a semigroup $S$ is \emph{inverse} if
$$\forall x\in S\ \exists!x^{-1}\in S\quad xx^{-1}x=x\text{ and }x^{-1}xx^{-1}=x^{-1}.$$
Each inverse semigroup admits a natural partial order
$$x\le y\iff xx^{-1}y=x.$$
If the semigroup is a lower semilattice with respect to this order relation then it can be considered an ai-semiring $(S,\wedge,\cdot)$ with the operation of taking greatest lower bounds as a semiring addition which follows from~\cite[Proposition~1.22]{Schein-73}. Following~\cite{Jackson-22}, we call such semirings \emph{naturally semilat\-tice-ordered inverse semigroups}.

\begin{problem}[{\!\cite[Problem~7.7(3)]{Jackson-22}}]\label{inv sem} 
Which finite naturally semilattice-ordered inverse semigroups are finitely based, in either of the signatures $\{+,\cdot\}$ or $\{+,\cdot,0\}$?
\end{problem}

We consider the signature $\{+,\cdot\}$. An example of a finite inverse semigroup which is nonfinitely based in the semiring signature is the Brandt monoid $B_2^1$. Recall that the Brandt semigroup $B_2$ is defined as the matrix semigroup $B_2=\{0,e_{11},e_{12},e_{21},e_{22}\}$ where
$$e_{11}=\begin{pmatrix}
1&0\\
0&0
\end{pmatrix},\ 
e_{12}=\begin{pmatrix}
0&1\\
0&0
\end{pmatrix},\ 
e_{21}=\begin{pmatrix}
0&0\\
1&0
\end{pmatrix},\ 
e_{22}=\begin{pmatrix}
0&0\\
0&1
\end{pmatrix}.$$

The monoid $B_2^1$ is the result of adjoining an identity element to $B_2$. This monoid is probably the most well-known example of a finite nonfinitely based semigroup~\cite{Perkins-68}. It remains nonfinitely based in the semi\-ring signature which was independently proven in~\cite{Jackson-22} and \cite{Volkov-21}. The semigroup $B_2$ is finitely based in semigroup signature. For a discussion of the history of this result, see~\cite{Volkov-19}. The finite basis problem in semiring signature for $B_2$ remains open. From this moment, we will always consider the semiring signature. 

Importance of the semirings $B_2$ and $B_2^1$ becomes clear within the context of the recent paper~\cite{Gusev-23} where Problem~\ref{inv sem} is considered for the class of finite combinatorial (i.e. not containing non-trivial subgroups) inverse semigroups. The main result of~\cite{Gusev-23} states that a finite combinatorial naturally semilattice-ordered inverse semigroup $S$ is a nonfinitely based ai-semiring whenever the inverse semigroup $B_2^1$ is contained in the variety of inverse semigroups generated by $S$. It follows from this result that the semigroup $B_2$ remains the only, up to equational equivalence, finite combinatorial naturally semilattice-ordered inverse semigroup whose finite basis problem (as an ai-semiring) remains open.

While studying the semiring $B_2$, it is natural to consider its subsemirings as well. All subsemirings of $B_2$ with $\le3$ elements are finitely based~\cite{Zhao-20}. It is easy to see that $B_2$ contains exactly two 4-element subsemirings, namely, $B_2\setminus\{e_{21}\}$ and $B_2\setminus\{e_{12}\}$ and these subsemi\-rings are isomorphic to each other. Therefore, from the point of view of identity bases it remains to consider the semiring $B_0=B_2\setminus\{e_{21}\}$.

\begin{problem}[{\!\cite[Problem~7.7(1)]{Jackson-22}}] Resolve the finite or nonfinite basabi\-lity of $B_0$ and $B_2$ as ai-semirings.
\end{problem}

The aim of the present paper is to prove that the semiring $B_0$ is finitely based. Throughout this text, we are working within the variety of all ai-semirings. Semigroup terms will be called \emph{words} and semiring terms \emph{polynomials}. For any polynomials $\mathbf p$ and $\mathbf q$, we use the notation $\mathbf p\le\mathbf q$ as an abbreviation for the identity $\mathbf p+\mathbf q\approx\mathbf p$. It is evident that this notation corresponds to the additive semilattice order relation. The following theorem is the main result of the article.

\begin{theorem}
\label{main}
The semiring $B_0$ has the following basis of identities within the variety of all ai-semirings:
\begin{align}
\label{periodicity}
&x^2\approx x^3,\\
\label{semigroup identities}
&x^2y^2\approx y^2x^2\approx xyx,\\
\label{square summand}
&x+y^2\approx x^2y^2,\\
\label{squaring eq}
&xy+x\approx xy^2,\ yx+x\approx y^2x,\\
\label{rook monoid}
&xy+zy+zt\le xt,\\
\label{crossing}
&x_1yz_1+x_2yz_2\le x_1yz_2
\end{align}
where $x_1,z_1,x_2,z_2$ can be empty.
\end{theorem}

In the statement of the theorem, the phrase ``$x_1,z_1,x_2,z_2$ can be empty'' means that the ``identity''~(\ref{crossing}) is to be understood as a system of 16 identities (some of them are trivial). Note that the identities (\ref{periodicity}) and (\ref{semigroup identities}) constitute an identity basis for the multiplicative semigroup $B_0$ which easily follows from the result of~\cite{Edmunds-80}. The proof of Theorem~\ref{main} is given in the second section.

\section{Proof of Theorem~\ref{main}}

We start with some preliminary facts. Note that addition in the semiring $B_0$ can be defined by the following simple rule:
$$x+y=\begin{cases}
x\text{ if }x=y,\\
0\text{ otherwise.}
\end{cases}$$

The following principle describes multiplication in $B_0$:
$$x_1x_2\cdots x_n=\begin{cases}
e_{11}\text{ if }x_i=e_{11}\text{ for all }i,\\
e_{22}\text{ if }x_i=e_{22}\text{ for all }i,\\
e_{12}\text{ if }x_i=e_{12}\text{ for some }i\text{, }x_j=e_{11}\text{ for all }j<i,\\
\text{and }x_j=e_{22}\text{ for all }j>i,\\
0\text{ otherwise.}
\end{cases}$$

The following facts are evident and will be used without direct reference:
\begin{align*}
&\mathbf p\le\mathbf q\text{ and }\mathbf q\le\mathbf r\text{ imply }\mathbf p\le\mathbf r,\\
&\mathbf p\le\mathbf q\text{ implies }\mathbf p+\mathbf r\le\mathbf q+\mathbf r,\\
&\mathbf p\le\mathbf q\text{ implies }\mathbf p\mathbf r\le\mathbf q\mathbf r,\\
&\mathbf p\le\mathbf q\text{ implies }\mathbf r\mathbf p\le\mathbf r\mathbf q
\end{align*}
where $\mathbf p,\mathbf q,\mathbf r$ are arbitrary polynomials.

For convenience of references, note the following obvious principle:
\begin{equation}
\label{adding inequalities}\mathbf p\le\mathbf q\text{ and }\mathbf p\le\mathbf  r\text{ imply }\mathbf p\le\mathbf q+\mathbf r.
\end{equation}

The identity~(\ref{squaring eq}) implies
\begin{equation}
\label{deleting squares} xy^2\le x,\ y^2x\le x.
\end{equation}

By $\Sigma$ we denote the system of identities~(\ref{periodicity})--(\ref{crossing}).

\begin{lemma} 
\label{sigma holds in b0} The semiring $B_0$ satisfies the system $\Sigma$.
\end{lemma}
\begin{proof} The identities~(\ref{periodicity})--(\ref{squaring eq}) contain just one or two variables, so they can be checked by enumeration of all values of variables. Consider the left hand side $xy+zy+zt$ of the identity~(\ref{rook monoid}). It is evident that a polynomial takes the value $e_{11}$ or $e_{22}$ if and only if all its variables take this value. Therefore, if the value of $xy+zy+zt$ is $e_{11}$ or $e_{22}$ then the value of $xt$ is the same. If $xy+zy+zt$ takes the value $0$ then the identity is trivial because $0$ is the zero element of the additive semilattice. It remains to assume that the value of $xy+zy+zt$ is $e_{12}$. Therefore, each summand $xy$, $zy$, and $zt$ should take the value $e_{12}$. This is possible in the following cases: 
\begin{itemize}
\item[a)] $x$ and $z$ take the value $e_{12}$; $y$ and $t$ take the value $e_{22}$;
\item[b)] $x$ and $z$ take the value $e_{11}$; $y$ and $t$ take the value $e_{12}$.
\end{itemize}
In both cases, the right hand side $xt$ also takes the value $e_{12}$. It remains to check~(\ref{crossing}). As with the identity~(\ref{rook monoid}), verification can be reduced to the case where the left hand side $x_1yz_1+x_2yz_2$ takes the value $e_{12}$. This is possible in the following cases:
\begin{itemize}
\item[a)] $x_1$ and $x_2$ take the value $e_{12}$; $y$, $z_1$ and $z_2$ take the value $e_{22}$ or are empty; 
\item[b)] $z_1$ and $z_2$ take the value $e_{12}$; $y$, $x_1$ and $x_2$ take the value $e_{11}$ or are empty; 
\item[c)] $y$ takes the value $e_{12}$; $x_1$ and $x_2$ take the value $e_{11}$ or are empty; $z_1$ and $z_2$ take the value $e_{22}$ or are empty. 
\end{itemize}
In all these cases, the right hand side $x_1yz_2$ also takes the value $e_{12}$.
\end{proof}

We write $\mathbf p\approx_{\Sigma}\mathbf q$ [$\mathbf p\le_\Sigma\mathbf q$] if the identity $\mathbf p\approx\mathbf q$ [$\mathbf p\le\mathbf q$] follows from $\Sigma$ within the variety of all ai-semirings. For a polynomial $\mathbf p$, we denote by $c(\mathbf p)$ the \emph{content} of $\mathbf p$, i.e. the set of all letters which occur in $\mathbf p$. For any two letters $x,y\in c(\mathbf p)$, we write $x\stackrel{\mathbf p}{\rightarrow}y$ if either $x$ coincides with $y$ or there exists a word $\mathbf w\ge_{\Sigma}\mathbf p$ such that at least one occurrence of $x$ precedes at least one occurrence of $y$ in $\mathbf w$. In this case, we have $\mathbf w=\mathbf w_1x\mathbf w_2y\mathbf w_3$ where the words $\mathbf w_1,\mathbf w_2,\mathbf w_3$ can be empty. We say that a letter $x\in c(\mathbf p)$ is \emph{rare in} $\mathbf p$ if, for each word $\mathbf w\ge_{\Sigma}\mathbf p$, the letter $x$ occurs in $\mathbf w$ at most once. This means that there is no word $\mathbf w_1x\mathbf w_2x\mathbf w_3\ge_{\Sigma}\mathbf p$, where $\mathbf w_1,\mathbf w_2,\mathbf w_3$ can be empty. As an example, consider the polynomial $\mathbf p=x^2y+yx^2$. We have $x\stackrel{\mathbf p}{\rightarrow}y$ and $y\stackrel{\mathbf p}{\rightarrow}x$ since $\mathbf p\le_{\Sigma}x^2y$ and $\mathbf p\le_{\Sigma}yx^2$ respectively. The letter $x$ is not rare in $\mathbf p$ since $\mathbf p\le_{\Sigma}x^2y$. Neither is the letter $y$ rare in $\mathbf p$ because $\mathbf p=x^2y+yx^2\le_{\Sigma}x^2yx^2\approx_{\Sigma}x^4y^2$ by~(\ref{crossing}) and~(\ref{semigroup identities}). In the further text, we implicitly assume that $\mathbf w_1,\mathbf w_2,\mathbf w_3$ are allowed to be empty each time we refer to the definition of the relation $\stackrel{\mathbf p}{\rightarrow}$ or to the definition of a rare letter. The similar agreement will be meant each time we consider an occurrence of a letter $x$ in a word $\mathbf w_1x\mathbf w_2$. The set of all rare letters in a polynomial $\mathbf p$ will be denoted by $R(\mathbf p)$. 

\begin{lemma}
\label{order} For any polynomial $\mathbf p$, the following holds.
\begin{itemize} 
\item[\textup{1)}] The relation $\stackrel{\mathbf p}{\rightarrow}$ is a quasiorder relation on $c(\mathbf p)$ (i.e. is reflexive and transitive).
\item[\textup{2)}] If $x\stackrel{\mathbf p}{\rightarrow}y$ and $y\stackrel{\mathbf p}{\rightarrow}x$ for some $x\in R(\mathbf p)$, $y\in c(\mathbf p)$, then $x=y$. In particular, the relation ${\stackrel{\mathbf p}{\rightarrow}}|_{R(\mathbf p)}$ is an order relation on $R(\mathbf p)$.
\end{itemize}
\end{lemma}
\begin{proof} 
1) The relation $\stackrel{\mathbf p}{\rightarrow}$ is reflexive by definition. Let us prove transitivity. Suppose $x\stackrel{\mathbf p}{\rightarrow}y$ and $y\stackrel{\mathbf p}{\rightarrow}z$. This means that there exist words $\mathbf u_1x\mathbf u_2y\mathbf u_3,\mathbf v_1y\mathbf v_2z\mathbf v_3\ge_{\Sigma}\mathbf p$. By~(\ref{adding inequalities}) and~(\ref{crossing}) we have
$$\mathbf p\le_{\Sigma}\mathbf u_1x\mathbf u_2y\mathbf u_3+\mathbf v_1y\mathbf v_2z\mathbf v_3\le_{\Sigma}\mathbf u_1x\mathbf u_2y\mathbf v_2z\mathbf v_3$$
whence $x\stackrel{\mathbf p}{\rightarrow}z$.

2) There exist words $\mathbf u_1x\mathbf u_2y\mathbf u_3,\mathbf v_1y\mathbf v_2x\mathbf v_3\ge_{\Sigma}\mathbf p$. By~(\ref{adding inequalities}) and~(\ref{crossing}), we have
$$\mathbf p\le_{\Sigma}\mathbf u_1x\mathbf u_2y\mathbf u_3+\mathbf v_1y\mathbf v_2x\mathbf v_3\le_{\Sigma}\mathbf u_1x\mathbf u_2y\mathbf v_2x\mathbf v_3$$
whence $x\not\in R(\mathbf p)$, a contradiction.
\end{proof}

For a word $\mathbf w$, we denote by $\mathbf w^{(2)}$ the result of letterwise squaring of $\mathbf w$. In other words, $(x_1x_2\cdots x_n)^{(2)}=x_1^2x_2^2\cdots x_n^2$.

\begin{lemma}
\label{squaring lemma} Consider a polynomial $\mathbf p$ and a word $\mathbf w_1\mathbf w_2\mathbf w_3\ge_{\Sigma}\mathbf p$ where $\mathbf w_1$ and $\mathbf w_3$ can be empty. If $\mathbf w_2$ contains no letters which are rare in $\mathbf p$ then $\mathbf p\le_{\Sigma}\mathbf w_1\mathbf w_2^{(2)}\mathbf w_3$.
\end{lemma}
\begin{proof}
Consider an occurrence of a letter $x$ in $\mathbf w_2$: $\mathbf w_2=\mathbf sx\mathbf t$.  Since $x\not\in R(\mathbf p)$, there exists a word $\mathbf u_1x\mathbf u_2x\mathbf u_3\ge_{\Sigma}\mathbf p$. Hence
\begin{align*}
\mathbf p&\le_{\Sigma}\mathbf w_1\mathbf sx\mathbf t\mathbf w_3+\mathbf u_1x\mathbf u_2x\mathbf u_3&&\text{by (\ref{adding inequalities})}\\
&\approx_{\Sigma}\mathbf w_1\mathbf sx\mathbf t\mathbf w_3+\mathbf u_1x^2\mathbf u_2^2\mathbf u_3&&\text{by (\ref{semigroup identities})}\\
&\le_{\Sigma}\mathbf w_1\mathbf sx\mathbf t\mathbf w_3+\mathbf w_1\mathbf sx^2\mathbf u_2^2\mathbf u_3&&\text{by (\ref{crossing}) and (\ref{adding inequalities})}\\
&\le_{\Sigma}\mathbf w_1\mathbf sx^2\mathbf t\mathbf w_3.&&\text{by (\ref{crossing})}
\end{align*}
Applying this argument for each occurrence of each letter in $\mathbf w_2$, we obtain the statement of the lemma.
\end{proof}

We say that a polynomial $\mathbf p$ is \emph{degenerate} if $R(\mathbf p)=\varnothing$. 

\begin{lemma}
\label{max chain} For an arbitrary non-degenerate polynomial $\mathbf p$ and an arbitrary word $\mathbf w\ge_{\Sigma}\mathbf p$, the set of letters $R(\mathbf p)\cap c(\mathbf w)$ is a maximal chain in the partially ordered set $(R(\mathbf p),{\stackrel{\mathbf p}{\rightarrow}}|_{R(\mathbf p)})$.
\end{lemma}
\begin{proof} First, we will prove that $R(\mathbf p)\cap c(\mathbf w)\neq\varnothing$. Suppose the contrary. By Lemma~\ref{squaring lemma}, we have $\mathbf p\le\mathbf w^{(2)}$. For any word $\mathbf v$ in $\mathbf p$, we have
\begin{align*}
\mathbf p&\le_{\Sigma}\mathbf v+\mathbf w^{(2)}&&\text{by (\ref{adding inequalities})}\\
&\approx_{\Sigma}\mathbf v+(\mathbf w^{(2)})^2&&\text{by (\ref{periodicity}) and (\ref{semigroup identities})}\\
&\approx_{\Sigma}\mathbf v^2(\mathbf w^{(2)})^2.&&\text{by (\ref{square summand})}
\end{align*}
This means that no letter of $\mathbf v$ is rare in $\mathbf p$. This holds for each word $\mathbf v$ in $\mathbf p$, so no letter in $\mathbf p$ is rare. Therefore, the polynomial $\mathbf p$ is degenerate, a contradiction.

It is obvious that $R(\mathbf p)\cap c(\mathbf w)$ is a chain. Suppose this chain is not maximal, so there exists a letter $x\in R(\mathbf p)\setminus c(\mathbf w)$ such that $(R(\mathbf p)\cap c(\mathbf w))\cup\{x\}$ is also a chain. There are three possible cases.

\textbf{Case~1.} The letter $x$ is the minimal element of the chain $(R(\mathbf p)\cap c(\mathbf w))\cup\{x\}$. Let $y$ be the minimal element of the chain $R(\mathbf p)\cap c(\mathbf w)$. The word $\mathbf w$ has the form $\mathbf w_1y\mathbf w_2$, where $\mathbf w_1$ contains no letters which are rare in $\mathbf p$. By Lemma~\ref{squaring lemma}, we have $\mathbf p\le_{\Sigma}\mathbf w_1^{(2)}y\mathbf w_2$. On the other hand, the condition $x\stackrel{\mathbf p}{\rightarrow}y$ means that there is a word $\mathbf u_1x\mathbf u_2y\mathbf u_3\ge_{\Sigma}\mathbf p$. Therefore,
\begin{align*}
\mathbf p&\le_{\Sigma}\mathbf w_1^{(2)}y\mathbf w_2+\mathbf u_1x\mathbf u_2y\mathbf u_3&&\text{by (\ref{adding inequalities})}\\
&\le_{\Sigma}\mathbf w_1^{(2)}y\mathbf w_2+\mathbf u_1x\mathbf u_2y\mathbf w_2&&\text{by (\ref{crossing}) and (\ref{adding inequalities})}\\
&\le_{\Sigma}y\mathbf w_2+\mathbf u_1x\mathbf u_2y\mathbf w_2&&\text{by (\ref{deleting squares})}\\
&\approx_{\Sigma}(\mathbf u_1x\mathbf u_2)^2y\mathbf w_2&&\text{by (\ref{squaring eq})}
\end{align*}
whence $x\not\in R(\mathbf p)$, a contradiction.

\textbf{Case~2.} The letter $x$ is the maximal element of $(R(\mathbf p)\cap c(\mathbf w))\cup\{x\}$. This case is dual to the previous one.

\textbf{Case~3.} The letter $x$, as an element of the chain $(R(\mathbf p)\cap c(\mathbf w))\cup\{x\}$, covers an element $y$ and is covered by an element $z$. Hence $z$ covers $y$ in the chain $R(\mathbf p)\cap c(\mathbf w)$. Therefore, $\mathbf w$ has the form $\mathbf w_1y\mathbf w_2z\mathbf w_3$ where $c(\mathbf w_2)\cap R(\mathbf p)=\varnothing$. By Lemma~\ref{squaring lemma}, we have $\mathbf p\le_{\Sigma}\mathbf w_1y\mathbf w_2^{(2)}z\mathbf w_3$. Furthermore, $y\stackrel{\mathbf p}{\rightarrow}x\stackrel{\mathbf p}{\rightarrow}z$, so there exist words $\mathbf u_1y\mathbf u_2x\mathbf u_3,\mathbf v_1x\mathbf v_2z\mathbf v_3\ge_{\Sigma}\mathbf p$. Hence
\begin{align*}
\mathbf p&\le_{\Sigma}\mathbf w_1y\mathbf w_2^{(2)}z\mathbf w_3+\mathbf u_1y\mathbf u_2x\mathbf u_3+\mathbf v_1x\mathbf v_2z\mathbf v_3&&\text{by (\ref{adding inequalities})}\\
&\le_{\Sigma}\mathbf w_1y\mathbf w_2^{(2)}z\mathbf w_3+\mathbf u_1y\mathbf u_2x\mathbf v_2z\mathbf v_3&&\text{by (\ref{crossing})}\\
&\le_{\Sigma}\mathbf w_1yz\mathbf w_3+\mathbf u_1y\mathbf u_2x\mathbf v_2z\mathbf v_3&&\text{by (\ref{deleting squares})}\\
&\le_{\Sigma}\mathbf w_1yz\mathbf w_3+\mathbf w_1y\mathbf u_2x\mathbf v_2z\mathbf v_3&&\text{by (\ref{crossing}) and (\ref{adding inequalities})}\\
&\le_{\Sigma}\mathbf w_1yz\mathbf w_3+\mathbf w_1y\mathbf u_2x\mathbf v_2z\mathbf w_3&&\text{by (\ref{crossing}) and (\ref{adding inequalities})}\\
&\approx_{\Sigma}\mathbf w_1y(z\mathbf w_3+\mathbf u_2x\mathbf v_2z\mathbf w_3)&&\text{the distributive law}\\
&\approx_{\Sigma}\mathbf w_1y(\mathbf u_2x\mathbf v_2)^2z\mathbf w_3&&\text{by (\ref{squaring eq})}
\end{align*}
whence $x\not\in R(\mathbf p)$, a contradiction.
\end{proof}

\begin{lemma}
\label{finite poset} In the ordered set $(R(\mathbf p),{\stackrel{\mathbf p}{\rightarrow}}|_{R(\mathbf p)})$, each maximal chain has non-empty intersection with each maximal antichain.
\end{lemma}
\begin{proof} 
Let $C$ be a maximal chain and $A$ a maximal antichain. Suppose $C\cap A=\varnothing$. Consider the least element $x_0$ of $C$. There are no elements $y\in A$ with $y\stackrel{\mathbf p}{\rightarrow} x_0$ since $C$ is maximal. The element $x_0$ can not be incomparable with all elements of $A$ since $A$ is maximal. Therefore, there exists at least one element $y_0$ in $A$ such that $x_0\stackrel{\mathbf p}{\rightarrow} y_0$. Let $x$ be the maximal element in $C$ such that $x\stackrel{\mathbf p}{\rightarrow}y$ for some $y\in A$. Since the chain $C$ is maximal, $x$ is not the absolute maximal element of $C$, so we can consider $z\in C$ which covers $x$ in $C$. By maximality of $C$, $z$ covers $x$ in $(R(\mathbf p),{\stackrel{\mathbf p}{\rightarrow}}|_{R(\mathbf p)})$. Since $A\cup\{z\}$ is not an antichain, the element $z$ is comparable with an element $t\in A$. The condition of maximality of $x$ excludes the case $z\stackrel{\mathbf p}{\rightarrow}t$ whence $t\stackrel{\mathbf p}{\rightarrow}z$. The conditions $x\stackrel{\mathbf p}{\rightarrow}y$, $x\stackrel{\mathbf p}{\rightarrow}z$, and $t\stackrel{\mathbf p}{\rightarrow}z$ mean that $\mathbf p\le_{\Sigma}\mathbf u_1x\mathbf u_2y\mathbf u_3,\mathbf v_1x\mathbf v_2z\mathbf v_3,\mathbf w_1t\mathbf w_2z\mathbf w_3$ for some $\mathbf u_1,\mathbf u_2,\mathbf u_3,\mathbf v_1,\mathbf v_2,\mathbf v_3,\mathbf w_1,\mathbf w_2,\mathbf w_3$. Since $z$ covers $x$, we have $c(\mathbf v_2)\cap R(\mathbf p)=\varnothing$. Hence, by Lemma~\ref{squaring lemma}, $\mathbf p\le_{\Sigma}\mathbf v_1x\mathbf v_2^{(2)}z\mathbf v_3$. Now we have
\begin{align*}
\mathbf u&\le_{\Sigma}\mathbf u_1x\mathbf u_2y\mathbf u_3+\mathbf v_1x\mathbf v_2^{(2)}z\mathbf v_3+\mathbf w_1t\mathbf w_2z\mathbf w_3&&\text{by (\ref{adding inequalities})}\\
&\le_{\Sigma}\mathbf u_1x\mathbf u_2y\mathbf u_3+\mathbf v_1xz\mathbf v_3+\mathbf w_1t\mathbf w_2z\mathbf w_3&&\text{by (\ref{deleting squares})}\\
&\le_{\Sigma}\mathbf u_1x\mathbf u_2y\mathbf u_3+\mathbf u_1xz\mathbf v_3+\mathbf w_1t\mathbf w_2z\mathbf w_3&&\text{by (\ref{crossing}) and (\ref{adding inequalities})}\\
&\le_{\Sigma}\mathbf u_1x\mathbf u_2y\mathbf u_3+\mathbf u_1xz\mathbf w_3+\mathbf w_1t\mathbf w_2z\mathbf w_3&&\text{by (\ref{crossing}) and (\ref{adding inequalities})}\\
&\le_{\Sigma}\mathbf w_1t\mathbf w_2\mathbf u_2y\mathbf u_3&&\text{by (\ref{rook monoid})}
\end{align*}
whence $t\stackrel{\mathbf p}{\rightarrow}y$. If $t\neq y$ then $A$ is not an antichain. If $t=y$ then $x\stackrel{\mathbf p}{\rightarrow}y=t\stackrel{\mathbf p}{\rightarrow}z$, whence $C\cup\{y\}$ is also a chain, so the chain $C$ is not maximal. In any case, a contradiction completes the proof.
\end{proof}

\begin{lemma} 
\label{isomorphism}
If $B_0$ satisfies an identity $\mathbf p\approx\mathbf q$ then
\begin{itemize}
\item[\textup{1)}] $c(\mathbf p)=c(\mathbf q)$;
\item[\textup{2)}] $R(\mathbf p)=R(\mathbf q)$;
\item[\textup{3)}] the relations $\stackrel{\mathbf p}{\rightarrow}$ and $\stackrel{\mathbf q}{\rightarrow}$ coincide.
\end{itemize}
\end{lemma}
\begin{proof} 1) Suppose $c(\mathbf p)\neq c(\mathbf q)$. Without loss of generality, suppose $c(\mathbf p)\not\subseteq c(\mathbf q)$. We assign the value $e_{11}$ to all letters in $c(\mathbf q)$ and the value $0$ to all letters in $c(\mathbf p)\setminus c(\mathbf q)$. It is clear that $\mathbf p$ takes the value $0$ and $\mathbf q$ takes the value $e_{11}$ whence the identity $\mathbf p\approx\mathbf q$ fails in $B_0$.

In the remaining part of the proof, for a non-degenerate polynomial $\mathbf p$, we need the function $\val_{\mathbf p,A}\colon c(\mathbf p)\rightarrow B_0$ associated with an arbitrary maximal antichain $A$ in the ordered set $(R(\mathbf p),{\stackrel{\mathbf p}{\rightarrow}}|_{R(\mathbf p)})$. This function is defined as follows:
$$\val_{\mathbf p,A}(x)=\left\{\begin{aligned}
&e_{12}\text{ if }x\in A,\\
&e_{11}\text{ if }x\not\in A\text{ and }x\stackrel{\mathbf p}{\rightarrow}y\text{ for some }y\in A,\\
&e_{22}\text{ otherwise.}
\end{aligned}\right.$$

The function $\val_{\mathbf p,A}$ can be naturally extended to all polynomials over the alphabet $c(\mathbf p)$. Let us prove that $\val_{\mathbf p,A}(\mathbf p)=e_{12}$. Take an arbitrary monomial $\mathbf w$ in $\mathbf p$. By Lemma~\ref{max chain}, the set $R(\mathbf p)\cap c(\mathbf w)$ is a maximal chain in the ordered set $(R(\mathbf p),{\stackrel{\mathbf p}{\rightarrow}}|_{R(\mathbf p)})$. By Lemma~\ref{finite poset}, this maximal chain contains a common element $y$ with $A$. Since $y\in R(\mathbf p)$, this letter occurs in $\mathbf w$ exactly once. By definition of the function $\val_{\mathbf p,A}$, we have $\val_{\mathbf p,A}(y)=e_{12}$ and $\val_{\mathbf p,A}(x)=e_{11}$ for each letter $x$ which precedes $y$ in $\mathbf w$. Let us prove that $\val_{\mathbf p,A}(x)=e_{22}$ for each $x$ which succeeds $y$ in $\mathbf w$. We have $y\stackrel{\mathbf p}{\rightarrow}x$ by definition of ${\stackrel{\mathbf p}{\rightarrow}}$. Since $A$ is an antichain, this implies $x\not\in A$. Therefore, $\val_{\mathbf p,A}(x)\neq e_{12}$. Suppose $\val_{\mathbf p,A}(x)=e_{11}$. Therefore, $x\stackrel{\mathbf p}{\rightarrow}z$ for some $z\in A$. Therefore, $y\stackrel{\mathbf p}{\rightarrow}z$ whence $y=z$ because $A$ is an antichain and $y,z\in A$. Therefore, $y\stackrel{\mathbf p}{\rightarrow}x\stackrel{\mathbf p}{\rightarrow}y$ which contradicts Lemma~\ref{order}.2). The contradiction shows that $\val_{\mathbf p,A}(x)=e_{22}$. Hence $\val_{\mathbf p,A}(\mathbf w)=e_{12}$. This holds for each word $\mathbf w$ in $\mathbf p$, so $\val_{\mathbf p,A}(\mathbf p)=e_{12}$.

It is easy to see that $\val_{\mathbf p,A}(x)=e_{22}$ if and only if $y\stackrel{\mathbf p}{\rightarrow}x$ for some $y\in A$.

Now we can return to the proof.

2) Suppose $R(\mathbf p)\neq R(\mathbf q)$. Without loss of generality, we can assume that there exists $x\in R(\mathbf p)\setminus R(\mathbf q)$. By the item 1) of this lemma, we have $c(\mathbf p)=c(\mathbf q)$. Take an arbitrary maximal antichain $A$ in the ordered set $(R(\mathbf p),{\stackrel{\mathbf p}{\rightarrow}}|_{R(\mathbf p)})$ such that $x\in A$. As was proven above, $\val_{\mathbf p,A}(\mathbf p)=e_{12}$. On the other hand, $x\not\in R(\mathbf q)$, whence there exists a word $\mathbf u_1x\mathbf u_2x\mathbf u_3\ge_{\Sigma}\mathbf q$. By definition, $\val_{\mathbf p,A}(x)=e_{12}$, so it is clear that $\val_{\mathbf p,A}(\mathbf u_1x\mathbf u_2x\mathbf u_3)=0$. Since $\Sigma$ holds in $B_0$, we have $\val_{\mathbf p,A}(\mathbf q)\le\val_{\mathbf p,A}(\mathbf u_1x\mathbf u_2x\mathbf u_3)=0$, that is $\val_{\mathbf p,A}(\mathbf q)=0$. Hence the identity $\mathbf p\approx\mathbf q$ fails in $B_0$.

3) Suppose $c(\mathbf p)=c(\mathbf q)$ and $R(\mathbf p)=R(\mathbf q)$. We must prove that ${\stackrel{\mathbf p}{\rightarrow}}={\stackrel{\mathbf q}{\rightarrow}}$. We start with the case when $\mathbf p$ is degenerate. Since $R(\mathbf p)=R(\mathbf q)$, the polynomial $\mathbf q$ is degenerate too. Let us prove that the relation ${\stackrel{\mathbf p}{\rightarrow}}$ coincides with $c(\mathbf p)\times c(\mathbf p)$. Take two arbitrary letters $x,y\in c(\mathbf p)$. There exist words $\mathbf u_1x\mathbf u_2,\mathbf v_1y\mathbf v_2$ in $\mathbf p$. These words can coincide, which does not contradict the further argument. By Lemma~\ref{squaring lemma}, we have $\mathbf p\le_{\Sigma}(\mathbf u_1x\mathbf u_2)^{(2)}=\mathbf u_1^{(2)}x^2\mathbf u_2^{(2)}$ and, similarly, $\mathbf p\le_{\Sigma}\mathbf v_1^{(2)}y^2\mathbf v_2^{(2)}$. Hence
\begin{align*}
\mathbf p&\le_{\Sigma}\mathbf u_1^{(2)}x^2\mathbf u_2^{(2)}+\mathbf v_1^{(2)}y^2\mathbf v_2^{(2)}&&\text{by (\ref{adding inequalities})}\\
&\le_{\Sigma}x^2+y^2&&\text{by~(\ref{deleting squares})}\\
&\approx_{\Sigma}x^2y^2.&&\text{by~(\ref{square summand}) and (\ref{periodicity})}
\end{align*}
We see that $x\stackrel{\mathbf p}{\rightarrow}y$. This is true for any $x$ and $y$, so ${\stackrel{\mathbf p}{\rightarrow}}=c(\mathbf p)\times c(\mathbf p)$. The same argument shows that ${\stackrel{\mathbf q}{\rightarrow}}=c(\mathbf p)\times c(\mathbf p)$.

From this moment, we suppose $\mathbf p$ and $\mathbf q$ are non-degenerate. Let $x\stackrel{\mathbf p}{\rightarrow}y$, and $x\not\stackrel{\mathbf q}{\rightarrow}y$ for some $x,y\in c(\mathbf p)$. There are three possible cases.

\textbf{Case~1.} $x,y\in R(\mathbf p)$. Here we have two subcases.

\textbf{Subcase~1.1.} $y\stackrel{\mathbf q}{\rightarrow}x$. Take a maximal antichain $A$ in $(R(\mathbf p),{\stackrel{\mathbf p}{\rightarrow}}|_{R(\mathbf p)})$ which contains $x$. We have $\val_{\mathbf p,A}(x)=e_{12}$, $\val_{\mathbf p,A}(y)=e_{22}$, and $\val_{\mathbf p,A}(\mathbf p)=e_{12}$. Furthermore, there exists a word $\mathbf u_1y\mathbf u_2x\mathbf u_3\ge_{\Sigma}\mathbf q$. Since $\val_{\mathbf p,A}(y)=e_{22}$ and $\val_{\mathbf p,A}(x)=e_{12}$, we have $\val_{\mathbf p,A}(\mathbf u_1y\mathbf u_2x\mathbf u_3)=0$. Hence $\val_{\mathbf p,A}(\mathbf q)=0$, so $\mathbf p\approx\mathbf q$ fails in $B_0$.

\textbf{Subcase~1.2.} $x$ and $y$ are incomparable with respect to $\stackrel{\mathbf q}{\rightarrow}$. Take an arbitrary maximal antichain $A$ in the poset $(R(\mathbf q),{\stackrel{\mathbf q}{\rightarrow}}|_{R(\mathbf q)})$ which contains $x$ and $y$. We have $\val_{\mathbf q,A}(x)=\val_{\mathbf q,A}(y)=\val_{\mathbf q,A}(\mathbf q)=e_{12}$. Furthermore, $x\stackrel{\mathbf p}{\rightarrow}y$ means that there exists a word $\mathbf u_1x\mathbf u_2y\mathbf u_3\ge_{\Sigma}\mathbf p$. Since $\val_{\mathbf q,A}(x)=\val_{\mathbf q,A}(y)=e_{12}$, we have $\val_{\mathbf q,A}(\mathbf p)\le\val_{\mathbf q,A}(\mathbf u_1x\mathbf u_2y\mathbf u_3)=0$ whence $\val_{\mathbf q,A}(\mathbf p)=0$. Thus $\mathbf p\approx\mathbf q$ fails in $B_0$.

\textbf{Case~2.} Either $x\in R(\mathbf p)$ and $y\not\in R(\mathbf p)$ or $y\in R(\mathbf p)$ and $x\not\in R(\mathbf p)$. Since these two possibilities are dual to each other, it is sufficient to consider the latter one. Take the set
$$I=\{z\in R(\mathbf q)\mid x\not\stackrel{\mathbf q}{\rightarrow}z\}\cup\{x\}.$$
It is clear that $I$ is an ideal in $(R(\mathbf q),{\stackrel{\mathbf q}{\rightarrow}}|_{R(\mathbf q)})$. Let $M$ be the set of all maximal elements of $I$. In particular, we have $x\in M$. Note that $M$ is an antichain in $(R(\mathbf q),{\stackrel{\mathbf q}{\rightarrow}}|_{R(\mathbf q)})$. To prove that this antichain is maximal, take an element $z\in R(\mathbf q)\setminus M$. If $x\stackrel{\mathbf q}{\rightarrow}z$ then $M\cup\{z\}$ is not an antichain because it contains two comparable elements $x$ and $z$. If $x\not\stackrel{\mathbf q}{\rightarrow}z$ then $z\in I$. Hence $z\stackrel{\mathbf q}{\rightarrow}t$ for some $t\in M$. Hence $M\cup\{z\}$ is not an antichain because it contains two comparable elements $z$ and $t$.

Consider the function $\val_{\mathbf q,M}$. We have $\val_{\mathbf q,M}(\mathbf q)=\val_{\mathbf q,M}(x)=e_{12}$. Since $x\not\stackrel{\mathbf q}{\rightarrow}y$, we have $y\in I$. If $y\in M$ then $\val_{\mathbf q,M}(y)=e_{12}$. If $y\in I\setminus M$ then $y\stackrel{\mathbf q}{\rightarrow}t$ for some $t\in M$ whence $\val_{\mathbf q,M}(y)=e_{11}$. In any case, $\val_{\mathbf q,M}(y)\neq e_{22}$. Since $x\stackrel{\mathbf p}{\rightarrow}y$, there exists $\mathbf u_1x\mathbf u_2y\mathbf u_3\ge_{\Sigma}\mathbf p$. The conditions $\val_{\mathbf q,M}(x)=e_{12}$ and $\val_{\mathbf q,M}(y)\neq e_{22}$ imply $\val_{\mathbf q,M}(\mathbf p)\le\val_{\mathbf q,M}(\mathbf u_1x\mathbf u_2y\mathbf u_3)=0$. Hence $\val_{\mathbf q,M}(\mathbf p)=0\neq\val_{\mathbf q,M}(\mathbf q)$, so $\mathbf p\approx\mathbf q$ fails in $B_0$.

\textbf{Case~3.} $x,y\not\in R(\mathbf p)$. Since $x\stackrel{\mathbf p}{\rightarrow}y$, there exists a word $\mathbf u_1x\mathbf u_2y\mathbf u_3\ge_{\Sigma}\mathbf p$. If $\mathbf u_2$ contains at least one letter $z$ which is rare in $\mathbf p$ then $x\stackrel{\mathbf p}{\rightarrow}z\stackrel{\mathbf p}{\rightarrow}y$. Since Case~2 has been already considered, we have $x\stackrel{\mathbf q}{\rightarrow}z\stackrel{\mathbf q}{\rightarrow}y$. Now we can assume that $c(\mathbf u_2)\cap R(\mathbf p)=\varnothing$. There are three subcases.

\textbf{Subcase~3.1.} $c(\mathbf u_1)\cap R(\mathbf p)=\varnothing$. By Lemma~\ref{max chain}, the set $c(\mathbf u_1x\mathbf u_2y\mathbf u_3)\cap R(\mathbf p)$ is a maximal chain in $(R(\mathbf p),{\stackrel{\mathbf p}{\rightarrow}}|_{R(\mathbf p)})$. In particular, $c(\mathbf u_1x\mathbf u_2y\mathbf u_3)\cap R(\mathbf p)\neq\varnothing$. Hence $c(\mathbf u_3)\cap R(\mathbf p)\neq\varnothing$. Let $z$ be the first letter in $\mathbf u_3$ which is rare in $\mathbf p$, so $\mathbf u_3=\mathbf v_1z\mathbf v_2$ where $c(\mathbf v_1)\cap R(\mathbf p)=\varnothing$. Now we have $\mathbf p\le_{\Sigma}\mathbf u_1x\mathbf u_2y\mathbf v_1z\mathbf v_2$ where $c(\mathbf u_1x\mathbf u_2y\mathbf v_1)\cap R(\mathbf p)=\varnothing$. In particular, $z$ is the minimal element of the chain $c(\mathbf u_1x\mathbf u_2y\mathbf u_3)\cap R(\mathbf p)$. Since this chain is maximal, $z$ is a minimal element of the ordered set $(R(\mathbf p),{\stackrel{\mathbf p}{\rightarrow}}|_{R(\mathbf p)})$. We have $x\stackrel{\mathbf p}{\rightarrow}z$ and $y\stackrel{\mathbf p}{\rightarrow}z$. By Case~2, this implies $x\stackrel{\mathbf q}{\rightarrow}z$ and $y\stackrel{\mathbf q}{\rightarrow}z$. Hence $\mathbf q\le_{\Sigma}\mathbf w_1x\mathbf w_2z\mathbf w_3$ and $\mathbf q\le_{\Sigma}\mathbf s_1y\mathbf s_2z\mathbf s_3$. By Case~1, the ordered sets $(R(\mathbf p),{\stackrel{\mathbf p}{\rightarrow}}|_{R(\mathbf p)})$ and $(R(\mathbf q),{\stackrel{\mathbf q}{\rightarrow}}|_{R(\mathbf q)})$ coincide. Hence $z$ is a minimal element of $(R(\mathbf q),{\stackrel{\mathbf q}{\rightarrow}}|_{R(\mathbf q)})$. Hence $c(\mathbf w_1x\mathbf w_2)\cap R(\mathbf q)=c(\mathbf s_1y\mathbf s_2)\cap R(\mathbf q)=\varnothing$. Now Lemma~\ref{squaring lemma} implies $\mathbf q\le_{\Sigma}(\mathbf w_1x\mathbf w_2)^{(2)}z\mathbf w_3=\mathbf w_1^{(2)}x^2\mathbf w_2^{(2)}z\mathbf w_3$ and $\mathbf q\le_{\Sigma}(\mathbf s_1y\mathbf s_2)^{(2)}z\mathbf s_3=\mathbf s_1^{(2)}y^2\mathbf s_2^{(2)}z\mathbf s_3$. Therefore,
\begin{align*}
\mathbf q&\le_{\Sigma}\mathbf w_1^{(2)}x^2\mathbf w_2^{(2)}z\mathbf w_3+\mathbf s_1^{(2)}y^2\mathbf s_2^{(2)}z\mathbf s_3&&\text{by (\ref{adding inequalities})}\\
&\le_{\Sigma}x^2z\mathbf w_3+y^2z\mathbf s_3&&\text{by (\ref{deleting squares})}\\
&\le_{\Sigma}x^2z\mathbf w_3+y^2z\mathbf w_3&&\text{by (\ref{crossing}) and (\ref{adding inequalities})}\\
&\approx_{\Sigma}(x^2+y^2)z\mathbf w_3&&\text{the distributive law}\\
&\approx_{\Sigma}x^2y^2z\mathbf w_3.&&\text{by (\ref{square summand}) and (\ref{periodicity})}
\end{align*}
Therefore, $x\stackrel{\mathbf q}{\rightarrow}y$.

\textbf{Subcase~3.2.} $c(\mathbf u_3)\cap R(\mathbf p)=\varnothing$. This subcase is dual to the previous one.

\textbf{Subcase~3.3.} $c(\mathbf u_1)\cap R(\mathbf p)\neq\varnothing$ and $c(\mathbf u_3)\cap R(\mathbf p)\neq\varnothing$. Let $z$ be the last letter in $\mathbf u_1$ which is rare in $\mathbf p$. Let $t$ be the first letter in $\mathbf u_3$ which is rare in $\mathbf p$. We have $\mathbf u_1=\mathbf v_1z\mathbf v_2$ and $\mathbf u_3=\mathbf w_1t\mathbf w_2$ where $c(\mathbf v_2)\cap R(\mathbf p)=c(\mathbf w_1)\cap R(\mathbf p)=\varnothing$. Hence $\mathbf p\le_{\Sigma}\mathbf v_1z\mathbf v_2x\mathbf u_2y\mathbf w_1t\mathbf w_2$. Therefore, $z\stackrel{\mathbf p}{\rightarrow}x\stackrel{\mathbf p}{\rightarrow}y\stackrel{\mathbf p}{\rightarrow}t$. By Case~2, this implies $z\stackrel{\mathbf q}{\rightarrow}x\stackrel{\mathbf q}{\rightarrow}t$ and $z\stackrel{\mathbf q}{\rightarrow}y\stackrel{\mathbf q}{\rightarrow}t$. Hence there exist $\mathbf r_1z\mathbf r_2x\mathbf r_3$, $\mathbf r_4x\mathbf r_5t\mathbf r_6$, $\mathbf s_1z\mathbf s_2y\mathbf s_3$, $\mathbf s_4y\mathbf s_5t\mathbf s_6\ge_{\Sigma}\mathbf q$. By Lemma~\ref{max chain}, the set $c(\mathbf v_1z\mathbf v_2x\mathbf u_2y\mathbf w_1t\mathbf w_2)\cap R(\mathbf p)$ is a maximal chain in $(R(\mathbf p),{\stackrel{\mathbf p}{\rightarrow}}|_{R(\mathbf p)})$. Since $\mathbf v_2$, $\mathbf u_2$, and $\mathbf w_1$ do not contain letters from $R(\mathbf p)$, $t$ covers $z$ in this chain. Since the chain is maximal, $t$ covers $z$ in $(R(\mathbf p),{\stackrel{\mathbf p}{\rightarrow}}|_{R(\mathbf p)})$. Therefore, by Case~1, $t$ covers $z$ in $(R(\mathbf q),{\stackrel{\mathbf q}{\rightarrow}}|_{R(\mathbf q)})$. Hence the words $\mathbf r_2,\mathbf r_5,\mathbf s_2,\mathbf s_5$ do not contain letters from $R(\mathbf q)$. By Lemma~\ref{squaring lemma}, we have $\mathbf q\le_{\Sigma}\mathbf r_1z\mathbf r_2^{(2)}x^2\mathbf r_3,\mathbf r_4x^2\mathbf r_5^{(2)}t\mathbf r_6,\mathbf s_1z\mathbf s_2^{(2)}y^2\mathbf s_3,\mathbf s_4y^2\mathbf s_5^{(2)}t\mathbf s_6$. Hence
\begin{align*}
\mathbf q&\le_{\Sigma}\mathbf r_1z\mathbf r_2^{(2)}x^2\mathbf r_3+\mathbf r_4x^2\mathbf r_5^{(2)}t\mathbf r_6\\
&+\mathbf s_1z\mathbf s_2^{(2)}y^2\mathbf s_3+\mathbf s_4y^2\mathbf s_5^{(2)}t\mathbf s_6&&\text{by (\ref{adding inequalities})}\\
&\le_{\Sigma}\mathbf r_1zx^2\mathbf r_3+\mathbf r_4x^2t\mathbf r_6+\mathbf s_1zy^2\mathbf s_3+\mathbf s_4y^2t\mathbf s_6&&\text{by (\ref{deleting squares})}\\
&\le_{\Sigma}\mathbf r_1zx^2t\mathbf r_6+\mathbf s_1zy^2t\mathbf s_6&&\text{by (\ref{crossing})}\\
&\le_{\Sigma}\mathbf r_1zx^2t\mathbf r_6+\mathbf r_1zy^2t\mathbf s_6&&\text{by (\ref{crossing}) and (\ref{adding inequalities})}\\
&\le_{\Sigma}\mathbf r_1zx^2t\mathbf r_6+\mathbf r_1zy^2t\mathbf r_6&&\text{by (\ref{crossing}) and (\ref{adding inequalities})}\\
&\approx_{\Sigma}\mathbf r_1z(x^2+y^2)t\mathbf r_6&&\text{the distributive law}\\
&\approx_{\Sigma}\mathbf r_1zx^2y^2t\mathbf r_6.&&\text{by (\ref{square summand}) and (\ref{periodicity})}
\end{align*}
Therefore, $x\stackrel{\mathbf q}{\rightarrow}y$.
\end{proof} 

\begin{lemma}
\label{final} Each identity that satisfies the conditions \textup{1)}, \textup{2)}, and \textup{3)} of Lem\-ma~\ref{isomorphism} follows from $\Sigma$.
\end{lemma}
\begin{proof}
Consider an identity $\mathbf p\approx\mathbf q$ that satisfies the conditions. Take an arbitrary word $\mathbf w$ in $\mathbf p$. Let $\mathbf w_1$ be a prefix of $\mathbf w$, so that $\mathbf w=\mathbf w_1\mathbf w_2$. We will prove that there exists a word $\mathbf w_1\mathbf w'_2\ge_{\Sigma}\mathbf q$ for some $\mathbf w'_2$. We use induction on the length of $\mathbf w_1$.

\textbf{Induction base}: $\mathbf w_1=x$ where $x$ is a letter. Since $c(\mathbf p)=c(\mathbf q)$, the polynomial $\mathbf q$ contains a word $\mathbf u_1x\mathbf u_2$. Suppose $\mathbf p$ is non-degenerate. If $y$ is a letter in $\mathbf u_1$ then $y\stackrel{\mathbf q}{\rightarrow}x$ whence $y\stackrel{\mathbf p}{\rightarrow}x$. Since $x\mathbf w_2\ge_{\Sigma}\mathbf p$, Lemma~\ref{max chain} implies that the set $c(x\mathbf w_2)\cap R(\mathbf p)$ is a maximal chain in $(R(\mathbf p),{\stackrel{\mathbf p}{\rightarrow}}|_{R(\mathbf p)})$ and $x\stackrel{\mathbf p}{\rightarrow}z$ where $z$ is the minimal element of this chain. Hence $y\stackrel{\mathbf p}{\rightarrow}z$. Hence $y\not\in R(\mathbf p)$. If $\mathbf p$ is degenerate then the conclusion $y\not\in R(\mathbf p)$ is obvious. This holds for each letter $y$ in $\mathbf u_1$ whence $\mathbf q\le_{\Sigma}\mathbf u_1^{(2)}x\mathbf u_2\le_{\Sigma}x\mathbf u_2$ by Lemma~\ref{squaring lemma} and (\ref{deleting squares}).

\textbf{Induction step}: $\mathbf w_1=\mathbf w'_1x$ where $x$ is a letter and $\mathbf q\le_{\Sigma}\mathbf w'_1\mathbf w'_2$. Let $y$ be the last letter in $\mathbf w'_1$, so that $\mathbf w'_1=\mathbf w''_1y$. We have $y\stackrel{\mathbf p}{\rightarrow}x$ whence $y\stackrel{\mathbf q}{\rightarrow}x$. Hence there exists $\mathbf v_1y\mathbf v_2x\mathbf v_3\ge_{\Sigma}\mathbf q$. Let $z$ be a letter in $\mathbf v_2$. We have $y\stackrel{\mathbf q}{\rightarrow}z\stackrel{\mathbf q}{\rightarrow}x$ whence $y\stackrel{\mathbf p}{\rightarrow}z\stackrel{\mathbf p}{\rightarrow}x$. Suppose $\mathbf p$ is non-degenerate. Since $\mathbf w=\mathbf w''_1yx\mathbf w_2\ge_{\Sigma}\mathbf p$, Lemma~\ref{max chain} implies that $c(\mathbf w''_1yx\mathbf w_2)\cap R(\mathbf p)$ is a maximal chain in $(R(\mathbf p),{\stackrel{\mathbf p}{\rightarrow}}|_{R(\mathbf p)})$. Therefore, the set $(c(\mathbf w''_1yx\mathbf w_2)\cap R(\mathbf p))\cup\{z\}$ is not a chain in $(R(\mathbf p),{\stackrel{\mathbf p}{\rightarrow}}|_{R(\mathbf p)})$ whence $z\not\in R(\mathbf p)=R(\mathbf q)$. If $\mathbf p$ is degenerate then the conclusion $z\not\in R(\mathbf p)=R(\mathbf q)$ is obvious. This holds for each letter $z$ in $\mathbf v_2$. Hence, by Lemma~\ref{squaring lemma}, we have $\mathbf q\le_{\Sigma}\mathbf v_1y\mathbf v_2^{(2)}x\mathbf v_3$. Therefore,
\begin{align*}
\mathbf q&\le_{\Sigma}\mathbf w'_1\mathbf w'_2+\mathbf v_1y\mathbf v_2^{(2)}x\mathbf v_3&&\text{by (\ref{adding inequalities})}\\
&=\mathbf w''_1y\mathbf w'_2+\mathbf v_1y\mathbf v_2^{(2)}x\mathbf v_3\\
&\le_{\Sigma}\mathbf w''_1y\mathbf w'_2+\mathbf v_1yx\mathbf v_3&&\text{by (\ref{deleting squares})}\\
&\le_{\Sigma}\mathbf w''_1yx\mathbf v_3.&&\text{by (\ref{crossing})}
\end{align*}
The induction step is finished.

Applying the statement for the case $\mathbf w_1=\mathbf w$, we obtain that $\mathbf q\le_{\Sigma}\mathbf w\mathbf w'$. If $\mathbf p$ is non-degenerate then, by Lemma~\ref{max chain}, the set $c(\mathbf w)\cap R(\mathbf p)$ is a maximal chain in the ordered set $(R(\mathbf p),{\stackrel{\mathbf p}{\rightarrow}}|_{R(\mathbf p)})=(R(\mathbf q),{\stackrel{\mathbf q}{\rightarrow}}|_{R(\mathbf q)})$ whence $c(\mathbf w')\cap R(\mathbf q)=\varnothing$. If $\mathbf p$ is degenerate then the conclusion $c(\mathbf w')\cap R(\mathbf q)=\varnothing$ is obvious. Hence $\mathbf q\le_{\Sigma}\mathbf w\mathbf {w'}^{(2)}\le_{\Sigma}\mathbf w$ by Lemma~\ref{squaring lemma} and (\ref{deleting squares}). We have $\mathbf q\le_{\Sigma}\mathbf w$ for each summand $\mathbf w$ in $\mathbf p$. By (\ref{adding inequalities}) this implies $\mathbf q\le_{\Sigma}\mathbf p$. The same argument shows that $\mathbf p\le_{\Sigma}\mathbf q$ whence $\mathbf p\approx_{\Sigma}\mathbf q$.
\end{proof}

\emph{Proof of Theorem}~\ref{main}. The theorem immediately follows from Lemmas~\ref{sigma holds in b0}, \ref{isomorphism}, and \ref{final}.\qed

\section*{Acknowledgements}

The author thanks Prof. Mikhail Volkov for his remarks on the text.

\end{document}